\documentclass[12pt,a4paper]{article}
\usepackage[T1]{fontenc}
\usepackage{amsmath}
\usepackage{amssymb}
\usepackage{graphicx}
\usepackage{color}
\usepackage{bbold}
\usepackage{dsfont}

\renewcommand{\P}{{\rm P}}
\newcommand{\E}{{\rm E}}
\newcommand{\RR}{\mathbb{R}}

\newcommand{\diag}{\mathrm{diag}}

\newcommand{\ud}{\mathrm{d}}  

\renewcommand{\u}{{\mbox{\tiny $+$}}}
\renewcommand{\d}{{\mbox{\tiny $-$}}}

\newcommand{\EE}{\mathcal{E}}
\newcommand{\s}{\mathcal{S}}

\newcommand{\MM}{\mathcal{M}}

\newcommand{\vect}[1]{\boldsymbol #1}

\newcommand{\valpha}{\vect \alpha}
\newcommand{\vrho}{\vect \rho}
\newcommand{\vbeta}{\vect \beta}

\newcommand{\vomega}{\vect \omega}

\newcommand{\vmu}{\vect \mu}
\newcommand{\vnu}{\vect \nu}
\newcommand{\vsigma}{\vect \sigma}
\newcommand{\vone}{\vect 1}
\newcommand{\vzero}{\vect 0}

\newcommand{\vG}{\vect G}
\newcommand{\vm}{\vect m}

\newcommand{\wQ}{\widetilde Q}
\newcommand{\wM}{\widetilde M}
\newcommand{\wA}{\widetilde A}
\newcommand{\wrho}{\widetilde{\vect\rho}}

\newcommand{\vligne}[1]{\begin{bmatrix} #1 \end{bmatrix}}

\newcommand{\vare}{\varepsilon}

\newtheorem{defn}{Definition}[section]
\newtheorem{lem}[defn]{Lemma}

\newtheorem{thm}[defn]{Theorem}
\newtheorem{cor}[defn]{Corollary}
\newtheorem{rem}[defn]{Remark}
\newtheorem{ass}[defn]{Assumption}

\newcommand{\cqfd}{\hfill $\square$}
\newcommand{\qed}{\hfill $\square$}
\newenvironment{proof}{
      \noindent {\bf Proof }}{\qed
      \vspace{0.25\baselineskip}
}
\newcommand{\debproof}{\begin{proof}}
\newcommand{\finproof}{\end{proof}}

\definecolor{darkmagenta}{rgb}{0.5,0,0.5}
\definecolor{darkgreen}{rgb}{0,0.6,0}
\definecolor{darkblue}{rgb}{0,0,0.6}
\definecolor{darkred}{rgb}{0.8,0,0}
\definecolor{mellow}{rgb}{.847, 0.72, 0.525}

\usepackage{pifont}

\begin{document}

\title{
Slowing time: Markov-modulated Brownian motion with a sticky boundary}

\author{Guy Latouche\thanks{Universit\'e Libre de Bruxelles,
D\'epartement d'informatique, CP 212, Boulevard du Triomphe, 1050
Bruxelles, Belgium, \texttt{latouche@ulb.ac.be}.} 
\and
Giang T. Nguyen\footnote{The University of Adelaide, School of Mathematical Sciences, SA 5005, Australia, \texttt{giang.nguyen@adelaide.edu.au}}
}
\date{\today}
\maketitle

\begin{abstract}
  We analyze the stationary distribution of regulated Markov modulated
  Brownian motions (MMBM) modified so that their evolution is slowed
  down when the process reaches level zero --- level zero is said to
  be {\em sticky}.  To determine the stationary distribution, we extend to
  MMBMs a construction of Brownian motion with sticky
  boundary, and we follow a Markov-regenerative approach similar to
  the one developed in past years in the context of
  quasi-birth-and-death processes and fluid queues.  We also rely on
  recent work showing that Markov-modulated Brownian motions may be
  analyzed as limits of a parametrized family of fluid queues.
We use our results to revisit the stationary distribution of the
well-known regulated MMBM.

\end{abstract}


\section{Introduction}
\label{s:intro}

Systems in real life are designed with feedback loops: a buffer in a
telecommunication network is not allowed to repeatedly overflow
without its input being throttled, water conservation measures are
taken before reservoirs get thoroughly dry, and so on.  This is our
reason for being interested in stochastic processes with {\em
  reactive} boundaries, that is, processes that change behaviour upon
hitting some boundary.  In the present paper, we focus on regulated
Markov modulated Brownian motions (regulated
MMBMs for short), with a {\em sticky boundary} at level 0.

An MMBM is a two-dimensional process $\{X(t), \varphi(t) : t \geq 0\}$
with $X(\cdot) \in  \RR$ and $\varphi(\cdot) \in \MM = \{1, \ldots, m\}$
with $m < \infty$.  The component $\{\varphi(t)\}$ is a continuous-time
Markov chain and controls as follows the evolution of $\{X(t)\}$:
\begin{equation*}
X(t) = \int_0^t \mu_{\varphi(s)} \, \ud s  + \int_0^t
\sigma_{\varphi(s)} \, \ud W(s).
\end{equation*}
Here, $\mu_i$ and $\sigma_i$, for $i \in \MM$, are real numbers with
$\sigma_i \geq 0$, and $\{W(t)\}$ is a standard Brownian motion
independent of $\{\varphi(t)\}$.  We call $\varphi(t)$ the phase at
time $t$ and $X(t)$ the level.

The process $\{Z(t), \varphi(t)\}$ regulated at zero is
\begin{equation}
   \label{e:Z}
   Z(t) = X(t)  + |\inf_{0 \leq s \leq t} X(s) |.
\end{equation}
We  assume that the MMBM is drifting to $-\infty$, so that $\{Z(t),
\varphi(t)\}$ has a stationary distribution;  this is made more precise
in Section~\ref{s:flip-flop} and we refer to Asmussen~\cite{asmus95b},
Rogers~\cite{roger94} for a general presentation of basic properties.
If $\sigma_i =0$ for all $i$, $\{Z(t), \varphi(t)\}$ is a fluid queue, a family of
processes extensively analyzed by Ramaswami~\cite{ram99}, da Silva
Soares and Latouche~\cite{dssl07} and Bean {\it et al.}~\cite{brt05},
among others.

Brownian motions with a sticky boundary were introduced by
Feller~\cite{felle52} in the 1950s.  Briefly stated, the regulated
Brownian motion is slowed down when it is at level 0, in such a way
that, without actually staying at zero for any interval of time of
positive length, it does spend in that level an amount of time with
positive Lebesgue measure.

The construction in Harrison and Lemoine~\cite{hl81} works as follows:
one starts with a Brownian motion $\{X^*(t): t \geq 0\}$ with parameters
$\mu$ and $\sigma^2$, define its regulator $R^*(t) = |\inf_{0 \leq
  s \leq t} X^*(s)|$ and define the regulated process as $Z^*(t) = X^*(t) +
R^*(t)$.  Next, one defines the functions $V^*(t) = t +
R^*(t)/\omega$, where $\omega > 0$ is some fixed constant, and
$\Gamma^*(t)$ such that $V^*(\Gamma^*(t))=t$.  Finally, one defines
\begin{equation}
   \label{e:bmhl}
Y^*(t) = Z^*(\Gamma^*(t)).
\end{equation}
We refer to $\Gamma^*$ as the new clock.
The process $\{Y^*(t)\}$ is a Brownian motion with sticky boundary, with
parameters $\mu$, $\sigma^2$, and $\omega$.
The time change in (\ref{e:bmhl}) modifies each trajectory of $Z^*$ by
increasing the time spent in state 0, while leaving $\{Y^*(t)\}$ to behave
exactly like $\{Z^*(t)\}$ away from zero.  

We extend in two ways this construction to Markov modulated Brownian
motion.  First, we define in Section~\ref{s:sticky-mmbms} a
straightforward generalization based on the regulator $R(t)$.  Taking
the phase into account, we decompose $R(t)$ as the sum of $m$
sub-regulators
\begin{equation}
   \label{e:ri}
r_i(t) = \int_0^t  {\mathbb 1}\{\varphi(s)=i\} \, \ud R(s)  
\end{equation}
for $i \in \MM$ and all $t$.  It is clear that $r_i(t)$ increases only when
$Z(t)=0$ and $\varphi(t)=i$, and that $R(t) = \sum_{i
  \in \MM} r_i(t)$.  Next, we define 
\begin{equation}
   \label{e:vomega}
V(t) = t + \sum_{i \in \MM} r_i(t)/\omega_i
\end{equation}
where $\omega_i >0$, and we define the function $\Gamma(t)$ such that
$V(\Gamma(t))=t$.  Observe that through the definition
(\ref{e:vomega}) of $V(t)$, we allow the clock $\Gamma(t)$ to slow
down at different rates for different phases.  Our new process is
$\{Y(t), \bar\varphi(t)\}$ with $Y(t)=Z(\Gamma(t))$,
$\bar\varphi(t)=\varphi(\Gamma(t))$.
 In Section~\ref{s:resampling} we investigate
another process, such that the marginal distribution of the phase is
allowed to change as a result of the process hitting the boundary.

To determine the stationary distribution of our processes, we
follow a Markov-regenerative approach.  We choose  points of regeneration
forming a subset of the epochs when the process hits level 0: let
$\{\Delta_n, n \geq 0\}$ denote a sequence of i.i.d. random variables
exponentially distributed with parameter~$q$.  We define
\begin{equation}
   \label{e:theta}
\theta_{n+1} = \inf\{ t > \theta_n + \Delta_{n+1} : Y(t) =0\},
\end{equation}
for $n \geq 0$, with $\theta_0 = 0$.  In short, once the process hits
the boundary, we start an exponential timer and we do not
register the instantaneous returns to 0 by the Brownian
motion; at the expiration of the timer, one is again
able to register the next hit at zero.
The process $\{\bar\varphi_n\}$ embedded at the regeneration epochs, with
$\bar\varphi_n = \bar\varphi(\theta_n)$, for $n \geq 1$, is an irreducible
discrete-time Markov chain with stationary
distribution  $\vrho$.  

We also define
\[
M_{ij}(x) = \E[\int_{\theta_n}^{\theta_{n+1}} \mathbb{1}\{Y(s) \in
[0,x], \bar\varphi(s)=j\} \,\ud s| \bar\varphi_n =i]
\]
independently of $n$; that is, $M_{ij}(x)$ is the expected sojourn
time of $\{Y(t), \bar\varphi(t)\}$ in $[0,x] \times \{j\}$ during an
inter-regeneration interval, given that the phase is $i$ at the
beginning of the interval.  The components of $\vm$, defined as $\vm =
M(\infty) \vone$, are the conditional expected lengths of intervals
between regeneration points, given the phase at the last regeneration:
its $i$th component is $m_i = \E[\theta_{n+1} - \theta_n |
\bar\varphi_n = i]$.  Finally, with
\[
G_i(x) = \lim_{t \rightarrow \infty} \P[ Y(t) \leq x, \bar\varphi(t)=i],
\]
and $\vG(x) = \vligne{G_1(x) & \ldots & G_m(x)}$, we have
\begin{equation}
   \label{e:G}
\vG(x) = (\vrho \vm)^{-1} \,\vrho M(x)
\end{equation}
(see \c{C}inlar~\cite[Section 10.7]{cinla75}).  With our
 choice (\ref{e:theta}) for the regeneration
points, $\vrho$ and $M$ are necessarily functions of the parameter
$q$, while the function $\vG$ is independent of $q$; as we show 
in  Theorem \ref{t:distribution}, the expression in the right-hand
side of (\ref{e:G}) is indeed independent of $q$.
We note that any set of regeneration points
will do, provided that they lead to a discrete-time Markov chain and
that one is able to determine the expected sojourn times between
regenerations.   

We rely on results obtained in Latouche and Nguyen~\cite{ln14,ln13} to
determine $\vrho$ and $M$.  We defined in \cite{ln13} a family of fast
oscillating fluid queues and we showed that MMBMs arise as limits
of such fluid queues as the speed of oscillation increases to
infinity.  We give in Section~\ref{s:flip-flop} the basic definition
of the approximating fluid queues as well as some properties that we
shall be using throughout the paper.  Before that, we show in
Section~\ref{s:stickyBM} how to construct a family of fast oscillating
fluid queues that converge to a Brownian motion with sticky boundary.

We determine in Section~\ref{s:sticky-mmbms} the stationary
distribution of our first family of Markov modulated processes with
sticky boundary.  In Section~\ref{s:regenerative}, we apply our
regenerative approach to the well-known regulated MMBM with one
boundary and obtain a new form for its stationary distribution; this
is discussed in Section~\ref{s:observation}, where we analyze the
physical meaning of our result and compare it to other expressions
available from the literature.  In Section~\ref{s:resampling}, we
define and analyze our second model for MMBMs with sticky boundary,
and we give some brief concluding remarks in
Section~\ref{s:conclusion}.

\paragraph{Notation}
We represent by $\vone$ a column vector of 1s and by $\vzero$ a vector of
0s.  We generally use the notation $X$ for unregulated processes, $Z$
for processes with one boundary, and $Y$ for processes with a sticky
boundary.  Also, a bar over a symbol, as in $\bar\varphi$, indicates
the phase process of the sticky version.

\section{Sticky Brownian motion}
\label{s:stickyBM}

As indicated in the introduction, we proceed in a manner similar to
Harrison and Lemoine~\cite{hl81}.   In this section, $\{X(t)\}$ denotes a Brownian
motion with parameters $\mu < 0$ and $\sigma^2 > 0$,
and the regulated process is $Z(t) = X(t) + R(t)$, where $R(t) = |\inf_{0
  \leq s \leq t} X(s)|$.  Our objective is to create a process that
behaves exactly like $Z$ when it is strictly positive but spends
more time at~0. To do this, one defines the function $V(t) = t +
R(t)/\omega$, where $\omega > 0$, and the function $\Gamma(t)$ such
that $V(\Gamma(t))=t$.  Finally, one defines
\[
Y(t) = Z(\Gamma(t)).
\]
The process $\{Y(t)\}$ is a sticky Brownian motion with parameters
$\mu$, $\sigma^2$ and~$\omega$.

The process $R(t)$ is non-negative, non-decreasing and continuous, and
remains constant when $Z(t) >0$.  Therefore, $V(t)$ is strictly
increasing and continuous, and $\Gamma(t)$ is well-defined, continuous
and strictly increasing.  The purpose of $\Gamma(t)$ is to serve as a
new clock, which increases at the same rate as $t$ when $Z(t) >0$, and
at a slower rate when $Z(t)=0$, that is, when $R(t)$ is increasing.
The stationary distribution of $Y$ is
\begin{equation}
   \label{e:bmsticky}
\lim_{t \rightarrow \infty} \P[Y(t) \leq x] =
\frac{|\mu|}{|\mu|+\omega} +
\frac{\omega}{|\mu|+\omega} (1-e^{2\mu x/\sigma^2})
\end{equation}
(\cite[p. 221]{hl81}), notice that it has a mass at zero.

As an approximation to $\{X(t)\}$, we define a family of two-state 
fluid queues $\{X_\lambda(t), \kappa_\lambda(t)\}$ indexed by
$\lambda$.  The generator  of the phase process
$\{\kappa_\lambda(t)\}$ is
\[
T = \left[ 
\begin{array}{rr}
-\lambda & \lambda \\ \lambda & - \lambda
\end{array}
\right]
\]
and the fluid rates are $c_1 = \mu + \sigma \sqrt\lambda$ and $c_2 =
\mu - \sigma \sqrt\lambda$.  It is shown in Ramaswami~\cite{ram99}
that $\{X_\lambda(t)\}$ converges weakly to $\{X(t)\}$ as $\lambda$
tends to~$\infty$.  We use the term {\em flip-flop} processes to
characterize $\{X_\lambda(t)\}$ and other fluid queues to be defined
in later sections,
as a shorthand reminder of their behavior: the fluid queue
switches steadily faster, as $\lambda$ increases, between two
increasing fluid rates.
We use the regulator $R_\lambda (t) = |\inf_{0 \leq s \leq t}
X_\lambda (s)|$ and the regulated process $Z_\lambda(t) =
X_\lambda(t) + R_\lambda(t)$ to define a two-state flip-flop fluid
queue $\{Y_\lambda(t)\}$ with a sticky boundary at zero.

The total time spent by $Z_\lambda(t)$ at level 0 is 
\[
L_\lambda^0 (t)
= \int_0^t  {\mathbb 1}\{Z_\lambda (s) =0\} \, \ud s.
\]  
It is a sum of individual intervals, each of which is exponentially
distributed with parameter $\lambda$.
To define the process $\{Y_\lambda (t)\}$, we change the behaviour of
the phase when the level is 0, and we assume that the intervals of
time spent there are now exponentially distributed with parameter $a
\sqrt\lambda$ instead of $\lambda$, for some $a >0$.  Equivalently,
the intervals of time spent at level 0 are stretched by a factor
$\sqrt\lambda /a$ .
The function that defines the new time is 
\[
V_\lambda(t)  = t -L_\lambda^0 (t) + L_\lambda^0 (t)
\frac{\sqrt\lambda}{a} .
\]
Clearly,  $L_\lambda^0(t) = R_\lambda(t) / (\sigma \sqrt\lambda -
\mu)$, so that 
\begin{equation}
   \label{e:vlambda}
 V_\lambda(t) = t + R_\lambda (t) \frac{\sqrt\lambda -a}{a(\sigma \sqrt\lambda -
  \mu)}. 
\end{equation}
Finally, the new clock is given by the function $\Gamma_\lambda (t)$
such that $V_\lambda(\Gamma_\lambda(t)) = t$, and the fluid queue with
sticky boundary is $\{Y_\lambda(t), \bar\kappa_\lambda(t)\}$, with
$Y_\lambda(t) = Z_\lambda (\Gamma_\lambda(t))$ and
$\bar\kappa_\lambda(t) = \kappa_\lambda(\Gamma_\lambda(t))$.

\begin{thm}
   \label{t:laztflip-flop}
   The processes $\{Y_\lambda(t)\}$ weakly converge to the sticky
   Brownian motion $\{Y(t)\}$ with parameters $\mu$, $\sigma^2$ and
   $\omega = a \sigma$.
\end{thm}
\begin{proof}
To simplify our presentation, we rewrite (\ref{e:vlambda}) as
$V_\lambda(t) = t + R_\lambda(t) a_\lambda$.  We also rewrite the
equation $V_\lambda(\Gamma_\lambda(t))=t$ as
\begin{equation}
   \label{e:gammal}
\Gamma_\lambda(t) + R_\lambda(\Gamma_\lambda(t))a_\lambda = t.
\end{equation}
  We know by \cite[Corollary 3.3]{ln13} that $\{Z_\lambda(t)\}$ weakly
  converges to $\{Z(t)\}$ as $\lambda \rightarrow \infty$, and so 
  $\{R_\lambda(t)\}$ weakly converges to $\{R(t)\}$.  In addition, the
  coefficient of $R_\lambda(t)$ in (\ref{e:vlambda}) converges to
  $1/(a\sigma)$.   Therefore, the finite-dimensional distribution of
  $\{Y_\lambda(t)\}$ converges to the finite-dimensional distribution of
  the Brownian motion $\{Y(t)\}$ with a sticky boundary and parameters $\mu$,
  $\sigma^2$ and $\omega = a \sigma$, and we need only to prove
  tightness.

By Billingsley~\cite[Theorem 7.3]{billi99} we need to prove that the
processes are tight at time 0, which is obvious  since they are
all equal to zero, and that for all $\vare$, $\eta$, there exist
$\delta^* >0$ and $\lambda^*$ such that 
\[
\P[\sup_{|s-t|\leq\delta}|Y_\lambda(t) - Y_\lambda(s)| \geq\vare] \leq\eta
\]
for all $\lambda> \lambda^*$ and $\delta< \delta^*$.  Now, assume
temporarily that $s \leq t$.  We have by~(\ref{e:gammal})
\begin{align*}
\Gamma_\lambda(t) -  \Gamma_\lambda(s)  &= t-s -
(R_\lambda(\Gamma_\lambda(t)) - R_\lambda(\Gamma_\lambda(s))) a_\lambda
\\
  & \leq t-s,  
\end{align*}
since both $R_\lambda$ and $\Gamma_\lambda$ are non-decreasing
functions and $a_\lambda >0$  for sufficiently large $\lambda$.  The reverse inequality holds if $s \geq t$ and so
$|\Gamma_\lambda(t) -  \Gamma_\lambda(s) | \leq |t-s|$ for all $s$, $t$.
Therefore,
\begin{align*}
\sup_{|s-t|\leq\delta} |Y_\lambda(t) - Y_\lambda(s)|
& =
\sup_{|s-t|\leq\delta} |Z_\lambda(\Gamma_\lambda(t)) -
Z_\lambda(\Gamma_\lambda(s))| 
\\  & \leq \sup_{|s-t|\leq\delta}  |Z_\lambda(t) - Z_\lambda(s)|
\end{align*}
and
\[
\P[\sup_{|s-t|\leq\delta} |Y_\lambda(t) - Y_\lambda(s)| \geq \vare]
\leq \P[\sup_{|s-t|\leq\delta} |Z_\lambda(t) - Z_\lambda(s)| \geq \vare].
\]
This completes the proof since we know by Ramaswami~\cite[Theorem
5]{ramas13} and Whitt~\cite[Corollary 7]{whitt70} that
$\{Z_\lambda(t)\}$ is tight.
\end{proof}

\begin{rem} \em It is necessary that the transition rate of $\kappa_\lambda$
  at level zero should grow like $\sqrt\lambda$.  To see this, let us
  assume that we use some general function $f(\lambda)$ instead of $a
  \sqrt\lambda$.  The stretching factor is $\lambda / f(\lambda)$ and
  (\ref{e:vlambda}) becomes
\begin{align*}
V_\lambda (t) & = t + L_\lambda^0 (t) ( \frac{\lambda}{f(\lambda)}-1)
\\
  & = t + R_\lambda (t)  \frac{\lambda-
    f(\lambda)}{(\sigma \sqrt\lambda - \mu) f(\lambda)}  \\
  & = t + R_\lambda (t)  \frac{1-
    f(\lambda)/ \lambda}{ (\sigma - \mu/
    \sqrt\lambda) f(\lambda) /\sqrt\lambda}.
\end{align*}
If $f(\lambda) /\sqrt\lambda \rightarrow 0$ as $\lambda \rightarrow \infty$, then $V_\lambda (t)
\rightarrow \infty$ for all $t$ such that $R_\lambda(t) >0$, and the
limit of $\{Y_\lambda(t)\}$ is 0 for all $t$.  On the contrary, if $f(\lambda) /\sqrt\lambda \rightarrow \infty$, then
$V_\lambda (t) \rightarrow t$, $Y_\lambda(t) \rightarrow Z(t)$,
and the limiting process is just the original regulated Brownian motion.
\end{rem}

\section{Preliminaries} 
\label{s:flip-flop}

We assume in the sequel that $\{\varphi(t)\}$ is an irreducible Markov
process with generator $Q$ and we make the following two assumptions.

\begin{ass}
   \label{a:sigma}
The variances $\sigma_i$ are strictly positive for all $i$ in $\MM$.
\end{ass}
This assumption allows us to significantly simplify our presentation.

\begin{ass}
   \label{a:drift}
   The stationary drift $\valpha \vmu$ is strictly negative, where
   $\vmu = \vligne{\mu_1 &\ldots &\mu_m}$ and $\valpha$ is the
   stationary probability vector of $\{\varphi(t)\}$, that is,
   $\valpha Q = \vzero$, $\valpha \vone = 1$.
\end{ass}
The inequality $\valpha \vmu < 0$ is the necessary and sufficient
condition for our regulated MMBMs to have a stationary probability
distribution.

We shall use the same approach as in \cite{ln13} and start our
analysis from a parametrized family of approximating fluid queues
driven by a two-dimensional phase process $\{\kappa_\lambda(t),
\varphi_\lambda(t)\}$ on the state space $\{1,2\}\times \MM$, with
generator
\begin{equation}
   \label{e:qstar}
Q^*_\lambda = \vligne{Q-\lambda I & \lambda I \\ \lambda I & Q -
  \lambda I}
\end{equation}
and fluid rate matrix 
\[
C^*(\lambda) = \vligne{D+ \sqrt\lambda \Theta & \\ & D- \sqrt\lambda \Theta },
\]
where $D = \diag(\mu_1, \ldots , \mu_m)$ and $\Theta = \diag(\sigma_1,
\ldots, \sigma_m)$.   Next, we define
\[
X_\lambda(t) = \int_0^t C^*_{\kappa_\lambda(s), \varphi_\lambda(s)}
(\lambda) \, \ud s 
\]
and $Z_\lambda(t) = X_\lambda(t) + |\inf_{0 \leq s \leq t}
X_\lambda(s)|$.  For $\lambda$ large enough, the rates $\mu_i +
\sigma_i \sqrt\lambda$ corresponding to $\kappa_\lambda = 1$ are all positive,
and the rates $\mu_i - \sigma_i \sqrt\lambda$ corresponding to $\kappa_\lambda
= 2$ are all negative.  It is shown in \cite{ln13} that the processes
$\{Z_\lambda(t), \varphi_\lambda(t)\}$ weakly converge to the MMBM
$\{Z(t), \varphi(t)\}$ as $\lambda \rightarrow \infty$.
In this paper, the phase process $\{\kappa_\lambda(t),
\varphi_\lambda(t)\}$ will define the evolution of our regulated
processes whenever the fluid level is strictly positive.  Different
rules will apply at level 0 for different processes, and will be
separately detailed in each case.

We use the same definition (\ref{e:theta}) for the regeneration points
of all processes $\{Z_\lambda(t), \kappa_\lambda(t),
\varphi_\lambda(t)\}$ and we omit to indicate that the $\theta_n$s depend on
$\lambda$, so as not to clutter the notation unduly.  
A key quantity for the analysis of fluid queues is the matrix of first
return probabilities to level 0, starting from level 0 in a phase with
strictly positive fluid rate.   Because different processes have
different behaviors at level 0 but the same behavior away from the
boundary, it will be useful to use the
sequence $\{\tau_n\}$ of first instants when the fluid starts
increasing away from level 0 after a regeneration epoch:
\[
\tau_n = \inf\{t > \theta_n : \kappa_\lambda(t) =1\},
\] 
for $n \geq 0$.
During an interval $(\tau_n, \theta_{n+1})$, we need to know at any
given time whether the current timer has expired or not.  For that
reason, we add a third phase component, named $\chi_\lambda$, with
$\chi_\lambda = 1$ if the timer has not expired yet, and $\chi_\lambda
= 2$ otherwise.  The two-dimensional phase $(\kappa_\lambda,
\varphi_\lambda)$ always evolves according to the transition matrix
(\ref{e:qstar}) and the transitions of $(\kappa_\lambda, \chi_\lambda,
\varphi_\lambda)$ are controlled, during an interval $(\tau_n,
\theta_{n+1})$, by the matrix
\begin{equation}
   \label{e:Tlambda}
T(\lambda) =
\left[ \begin{array}{cc|cc}
Q-\lambda I - q I & q I & \lambda I &   \\
& Q-\lambda I &  & \lambda I \\
\hline
\lambda I & & Q-\lambda I - q I & q I \\
 & \lambda I &  & Q-\lambda I
\end{array}
\right].
\end{equation}
At regeneration times, the new component $\chi_\lambda$
instantaneously switches from $\chi_\lambda(\theta_n^\d) = 2$ to
$\chi_\lambda(\theta_n^\u) = 1$, where $\chi_\lambda(\theta_n^\d) =
\lim_{t \uparrow \theta_n}\chi_\lambda(t)$ and
$\chi_\lambda(\theta_n^\u) = \lim_{t \downarrow
  \theta_n}\chi_\lambda(t)$.
The diagonal matrix of fluid rates is
\[
{C(\lambda)} = \vligne{ I_2 \otimes (D + \Theta \sqrt{\lambda})
 \\  &  I_2 \otimes (D - \Theta \sqrt{\lambda}) }
\]
where $I_2$ is the identity matrix of order 2, and we partition the
state space into the subsets
\[
\s_\u  = \{(1, \chi_\lambda, \varphi_\lambda) : \chi_\lambda \in
\{1,2\}, \varphi_\lambda \in \MM\} 
\]
{and}
\[
\s_\d  = \{(2, \chi_\lambda, \varphi_\lambda) : \chi_\lambda \in
\{1,2\}, \varphi_\lambda \in \MM\}. 
\]
We similarly partition the matrix $T(\lambda)$ as
\[
T(\lambda) = \vligne{T_{\u\u}(\lambda)  & T_{\u\d}(\lambda) \\
  T_{\d\u}(\lambda)  & T_{\d\d}(\lambda)  }
\]
and we write
$
{C}_\u =  I_2 \otimes (D + \Theta \sqrt{\lambda})
$
and 
$
{C}_\d = I_2 \otimes (D - \Theta \sqrt{\lambda}).
$

The matrix of first return probabilities, indexed by $\s_\u \times
\s_\d$, is denoted as $\Upsilon_\lambda$ and defined by
\begin{align}
   \nonumber
(\Upsilon_\lambda)_{k,i;k',j} =\P[\xi < \infty, \, &\chi_\lambda(\xi) = k',
\varphi_\lambda(\xi) = j   \\
   \label{e:Ul}
 &| Z_\lambda(0)=0,  \kappa_\lambda(0) =
1, \chi_\lambda(0)=k, \varphi_\lambda(0)=i ]
\end{align}
where $\xi = \inf\{t
  >0: Z_\lambda(t)=0\}$ is the first return time to level 0. 
It is well known (Rogers~\cite{roger94}) that $\Upsilon_\lambda$ is
the minimal non-negative solution of the Riccati equation
\[ {C}_\u^{-1} T_{\u\d}(\lambda) + {C}_\u^{-1}
T_{\u\u}(\lambda) X + X |{C}_\d|^{-1}
T_{\d\d}(\lambda) + X |{C}_\d|^{-1}
T_{\d\u}(\lambda) X = 0.
\]
One easily verifies that it has the structure
\[
\Upsilon_\lambda = \vligne{\Psi _\lambda (q) & \Psi^c_\lambda(q) \\ 0 & \Psi_\lambda}
\]
where 
\begin{itemize}
\item $\Psi _\lambda (q)$ is the probability matrix of returning
  to the original level {\em before} the exponential timer expires, it
  is the minimal non-negative solution of
\begin{align*}
\lambda (D+ &\sqrt\lambda\Theta)^{-1} 
+ (D+\sqrt\lambda\Theta)^{-1} (Q-\lambda I - q I) X
\\ &
+ X |D-\sqrt\lambda\Theta|^{-1} (Q- \lambda I - q I )
+ \lambda X |D - \sqrt\lambda\Theta|^{-1} X =0,
\end{align*}
\item 
$\Psi _\lambda = \Psi_\lambda(0)$ is the return probability matrix without any time
constraint, and 
\item $\Psi^c _\lambda (q)$ is the probability of returning to
  the original level {\em after} the exponential timer has expired, so
  that
\begin{equation}
   \label{e:barpsi}
\Psi^c_\lambda (q) = \Psi _\lambda -   \Psi _\lambda (q).
\end{equation}
\end{itemize}
At level 0, we need two transition matrices.  The first matrix
has entries
\begin{equation}
   \label{e:pzero}
   (P_{\lambda,0})_{ij} = \P[\tau_n > \theta_n + \Delta_{n+1}, \varphi_\lambda(\theta_n +\Delta_{n+1}) =
   j | \varphi_\lambda(\theta_n)=i], 
\end{equation}
for $i$, $j$ in $\MM$: these are the probabilities that the process
continuously remains at level 0 until the expiration of the
exponential timer, at which time $\varphi_\lambda =j$, given the
phase at time $\theta_n$  is $i$.  We need not specify the other phase
component as 
$\kappa_\lambda(\theta_n+ \Delta_{n+1}) =2$.  The second matrix is
\begin{equation}
   \label{e:pone}
(P_{\lambda,1})_{ij} = \P[\tau_n < \theta_n +
  \Delta_{n+1}, \varphi_\lambda(\tau_n) = j| \varphi_\lambda(\theta_n)=i],  
\end{equation}
for $i$, $j$ in $\MM$: the exponential timer has not yet gone off at time
$\tau_n$, the component $\kappa_\lambda$ switches from 2 to 1 and the component
$\chi_\lambda$ remains equal to~1.

The phase transition matrix at regenerative epochs is
$\Phi _\lambda$, with
\[
(\Phi _\lambda)_{ij} = \P[ \varphi_\lambda(\theta_{n+1}) = j | \varphi_\lambda(\theta_n)=i ].
\]
Again, we do not need to specify the remaining components of the phases:
$\kappa_\lambda(\theta_n) = 2$ since the fluid rate is negative at
that time and $\chi_\lambda(\theta_n^\u) = 1$ since a new timer
interval begins immediately after the regeneration.

\begin{lem}
   \label{t:Phi}
The transition matrix at the regeneration epochs is given by
\begin{equation}
   \label{e:phi}
\Phi_\lambda = I - (I- P_{\lambda,1} \Psi _\lambda (q))^{-1} (I
-P_{\lambda,0} - P_{\lambda,1} \Psi _\lambda),
\end{equation}
Its stationary probability vector $\vrho_\lambda$ is 
\begin{equation}
   \label{e:rnlambda}
\vect\rho _\lambda  = c \vect\nu _\lambda (I- P_{\lambda,1} \Psi _\lambda (q))
\end{equation}
for some scalar $c$, where $\vect\nu _\lambda$ is such that
\begin{equation}
   \label{e:nulambda}
   \vect\nu _\lambda (P_{\lambda,0} + P_{\lambda,1} \Psi _\lambda)  = \vect\nu _\lambda
, \qquad \qquad \vect\nu _\lambda \vone = 1.
\end{equation}
\end{lem}
\begin{proof}
The transition matrix satisfies the following equation,
\begin{equation}
   \label{e:phi1}
\Phi_\lambda = P_{\lambda,0} + P_{\lambda,1}  \Psi^c_\lambda(q) 
+ P_{\lambda,1} \Psi_\lambda(q) \Phi_\lambda.
\end{equation}
Indeed, at time $\theta_n$ the process is at level 0 and 
either it does not leave level 0 before the
timer expires (this corresponds to the first term), or it does leave
level zero and returns after the timer has expired (this is the second
term) or it leaves level 0 and returns before the timer has expired,
in which case we still have to wait for the next regeneration point
(this gives the third term).

Now, starting from level 0 and any phase in $\s_\u$, there is a
strictly positive probability that the process returns to level 0
after the timer has expired.  Thus, $\Psi^c_\lambda(q) \vone > \vzero$
or equivalently, $\Psi_\lambda(q) \vone < \vone$, so that $P_{\lambda,1}
\Psi_\lambda(q) $ is a strictly sub-stochastic matrix and $I -
P_{\lambda,1} \Psi_\lambda(q)$ is non-singular.  Thus, (\ref{e:phi1}) becomes
\[
  \Phi _\lambda
   = (I- P_{\lambda,1} \Psi _\lambda (q))^{-1} (P_{\lambda,0} +
  P_{\lambda,1} \Psi^c_\lambda (q) )  
\]
by (\ref{e:barpsi}).  We may rewrite the last equation as 
\[
\Phi _\lambda
  = (I- P_{\lambda,1} \Psi _\lambda (q))^{-1}
  (P_{\lambda,0} + P_{\lambda,1} (\Psi _\lambda - \Psi _\lambda
  (q)) ) 
\]
and (\ref{e:phi}) is proved.

The stationary probability vector $\vect \rho _\lambda$ is such that 
$
\vect\rho _\lambda \Phi _\lambda = \vect \rho _\lambda$,   or
\begin{equation}
   \label{e:rholambda}
\vect\rho _\lambda (I- P_{\lambda,1} \Psi _\lambda (q))^{-1} (I -
P_{\lambda,0} -  P_{\lambda,1} \Psi_\lambda) = \vzero,
\end{equation}
which proves (\ref{e:rnlambda}) as soon as we show that $\vnu_\lambda$
exists and is unique.

By Assumption \ref{a:drift}, the return time to level 0 is finite
a.s., so that $\Psi_\lambda \vone = \vone$ for all $\lambda$ and
$P_{\lambda,0} +  P_{\lambda,1} \Psi_\lambda$ is an irreducible
stochastic matrix, with a unique stationary probability vector
$\vnu_\lambda$.   This concludes the proof.
\end{proof}

\begin{rem} \em
   \label{r:rho}
It is obvious that the stationary distribution of the phase at epochs
of regeneration depends, through the transition matrices
$P_{\lambda,0}$ and $P_{\lambda,1}$, on the rules of evolution of the
phase when the fluid is at level  0.  In order to avoid confusion, we
use in the sequel the notation $\vrho_\lambda$ and its limit $\vrho$
for the  process associated with the MMBM with sticky boundary,
$\vrho^*_\lambda$ (with limit $\vrho^*$) for the flip-flop
process analyzed in Section~\ref{s:regenerative} and associated with
the traditional MMBM, and $\widetilde\vrho$ for the process analysed
in Section~\ref{s:resampling}, with sticky boundary and resampling of
the phase at level 0.
\end{rem}

\section{ MMBM with sticky boundary}
\label{s:sticky-mmbms}

As explained in the introduction, we start from the function $V(t)$
defined in~(\ref{e:vomega}), which is monotone and continuous.  We use
it to define the new clock $\Gamma$ such that $V(\Gamma(t)) = t$, and
to define the new process $\{Y(t), \bar\varphi(t)\}$ such that $Y(t) =
Z(\Gamma(t))$ and $\bar\varphi(t) = \varphi(\Gamma(t))$.

We call this process a
Markov-modulated Brownian motion with sticky boundary, with parameters
$\vmu$, $\vsigma$, $\vomega$ and $Q$.
To obtain its stationary distribution, we proceed in three steps: we
construct a family of approximating fluid queues, then we determine
the stationary distribution $\vrho$ at epochs of regeneration, and
finally we obtain the matrix $M(x)$ of expected time spent in $[0,x]$
during regeneration intervals.

We decompose the regulator $R_\lambda(t)= |\inf_{0 \leq s \leq t}
X_\lambda(s)|$ of the flip-flop fluid queue $\{X_\lambda(t),
\kappa_\lambda(t), \varphi_\lambda(t)\}$ into
its sub-regulators 
\begin{equation}
   \label{e:ril}
r_{\lambda,i}(t) = \int_0^t  {\mathbb 1}\{\varphi_\lambda(s)=i\} \, \ud R_\lambda(s),
\end{equation}
and we write $Z_\lambda$ as the sum $Z_\lambda = X_\lambda + \sum_i
r_{\lambda,i}$.   The total time spent by the process in phase $i$
during the interval $(0,t)$ is $U_{\lambda,i}(t) = \int_0^t {\mathbb
  1}\{\varphi(s) = i\} \, \ud s$, for $i \in \MM$.

We repeat for every phase the argument in
Section~\ref{s:stickyBM}: the time spent by $\{Z_\lambda(t)\}$ at
level~0 in phase $i$ is $r_{\lambda,i}(t) / |\mu_i-\sigma_i
\sqrt\lambda|$ and the functions that redefine time are
\begin{align*}
V_{\lambda,i}(t) & = U_{\lambda,i}(t) - r_{\lambda,i}(t) 
\frac{1}{|\mu_i-\sigma_i \sqrt\lambda|} + r_{\lambda,i}(t) 
  \frac{1}{|\mu_i-\sigma_i \sqrt\lambda| }
  \frac{\lambda + |Q_{ii}|}{a_i \sqrt\lambda + |Q_{ii}|} \\
  & = U_{\lambda,i}(t) + r_{\lambda,i}(t) 
  \frac{\lambda - a_i \sqrt\lambda}{(\sigma_i \sqrt\lambda -
    \mu_i)(a_i \sqrt\lambda + |Q_{ii}|)},
\end{align*}
for $i$ in $\MM$.
For $\lambda$ large enough, their sum is
\[
V_\lambda(t) = t + \sum_i r_{\lambda,i}(t) \frac{\lambda +
  O(\sqrt\lambda)}{a_i \sigma_i \lambda + O(\sqrt\lambda)}
\]
and converges to $V(t)$ defined in (\ref{e:vomega}) 
with $\omega_i = a_i \sigma_i$.  In matrix notation, we may write
\begin{equation}
\label{e:lazyV}
  V(t) = t + \vect r(t)  A^{-1} \Theta^{-1} \vone,
\end{equation}
where $A = \diag(a_1, \ldots a_m)$ and $a_i > 0$ for all $i$.
With the function $V_\lambda(t)$, we define the new clock
$\Gamma_\lambda$ such that $V_\lambda(\Gamma_\lambda(t))=t$, and the
new processes $\{Z_\lambda(\Gamma_\lambda(t)),
\kappa_\lambda(\Gamma_\lambda(t)),
\varphi_\lambda(\Gamma_\lambda(t)))$.  By \cite[Theorem 2.7]{ln13},
$\{Z_\lambda(t), \varphi_\lambda(t)\}$ weakly converge to $\{Z(t),
\varphi(t)\}$ and, by the continuity property of $\Gamma(t)$, it is
clear that the finite-dimensional distribution of
$\{Z_\lambda(\Gamma_\lambda(t)), \varphi_\lambda(\Gamma_\lambda(t))\}$
converge to the finite-dimensional distribution of $\{Y(t),
\bar\varphi(t)\}$.

The process $\{Z_\lambda(\Gamma_\lambda(t)),
\kappa_\lambda(\Gamma_\lambda(t)),
\varphi_\lambda(\Gamma_\lambda(t))\}$ is not very convenient, however,
as its definition does not conform to the usual parametrization of
fluid queues.  For that reason, we define another fluid queue, denoted
as $\{Y_\lambda(t), \bar\kappa_\lambda(t), \bar\varphi_\lambda(t)\}$.
This new process and $\{Z_\lambda(\Gamma_\lambda(t)),
\kappa_\lambda(\Gamma_\lambda(t)),
\varphi_\lambda(\Gamma_\lambda(t))\}$ are not pathwise identical 
but they have the same distribution.

The two-dimensional phase
$(\bar\kappa_\lambda,\bar\varphi_\lambda)$ is controlled by the
generator $Q^*_\lambda$ defined in (\ref{e:qstar}) as long as
$Y_\lambda$ is strictly positive.  When $Y_\lambda =0$, the
transition rates are given by the new matrix
\begin{equation}
   \label{e:q0}
Q_{\lambda,0} = \vligne{ \sqrt{\lambda} A & (1/\sqrt\lambda) A (Q-\lambda I)}.
\end{equation}
This means that while the process is in level 0,
\begin{itemize}
\item if $\varphi_\lambda = i$ and $\kappa_\lambda = 2$, intervals of
  time are stretched by a factor $\sqrt\lambda / a_i$,
\item there may be a change to $\varphi_\lambda = j$ with the new rate
  $\bar Q_{ij} = a_i Q_{ij} / \sqrt\lambda$ without changing
  $\kappa_\lambda$,
\item or the process may change to $\kappa_\lambda =1$ at the rate
  $a_i \sqrt\lambda$, and leave the level 0, without changing
  $\varphi_\lambda$.
\end{itemize}
We denote by $\vrho_\lambda$ the stationary distribution of
$\bar\varphi_\lambda$ at epochs of regeneration and by $M_\lambda(x)$
the matrix of conditional expected time spent by the fluid queue in
$[0,x]$.

\begin{lem}
   \label{t:stickyboundary}
The stationary distribution $\vrho_\lambda$  converges, as
$\lambda \rightarrow \infty$,  to the vector
$\vect\rho$ such that
\begin{equation}
   \label{e:rhost}
\vect\rho (q A^{-1} -\Theta U(q))^{-1} \Theta U  = \vzero,  \qquad \vect\rho
  \vone = 1.
\end{equation}
where  $U(q)$ is the unique solution of 
\begin{equation}
   \label{e:quadraticU}
\frac{1}{2} \Theta^{2}  X^2 + D X + (Q-q I) =0
\end{equation}
with eigenvalues of negative real parts, and where $U=U(0)$.  Both $U$ and
$U(q)$ are generators.  For $q > 0$, all eigenvalues of
$U(q)$ have strictly negative real parts; for $q = 0$, one
eigenvalue of $U$ is equal to 0, the others have strictly negative
real parts.

Furthermore,
\begin{equation}
   \label{e:rhonust}
\vect\rho ( q A^{-1} -\Theta U(q) )^{-1}  = c_1
\vect\nu
\end{equation}
for some scalar $c_1$, where 
\begin{equation}
   \label{e:nust}
\vect\nu \Theta U = \vzero,  \qquad \vect\nu \vone = 1.
\end{equation}
\end{lem}

\begin{proof}
The matrix $\Upsilon_\lambda$ defined in (\ref{e:Ul}) does not depend
on the behaviour of the fluid queue at level 0, so that
Lemma~\ref{t:Phi} applies and the stationary probability vector
$\vrho_\lambda$ of $\bar\varphi_\lambda(t)$ is
given by (\ref{e:rnlambda}).  

The matrix $P_{\lambda,0}$ defined in (\ref{e:pzero}) is given here by
\begin{align}
   \nonumber
 P_{\lambda,0}  & =  \int_0^\infty q e^{-qu} e^{A(Q-\lambda I)
   u/\sqrt\lambda} \, \ud u
\\
   \nonumber
   & = q (\sqrt\lambda A + qI - (1/\sqrt\lambda) A Q)^{-1}  
 \\
   \nonumber
   & = \frac{1}{\sqrt\lambda} q A^{-1}  (I + \frac{1}{\sqrt\lambda} q A^{-1}
 - \frac{1}{\lambda} A Q A^{-1})^{-1}  
 \\
    \label{e:p0st}
   &=  \frac{1}{\sqrt\lambda} q A^{-1} + O(1/\lambda)
 \end{align}
and similarly
\begin{align}
   \nonumber
 P_{\lambda,1}  & =  \int_0^\infty  e^{-qu} e^{A(Q-\lambda I)
   u/\sqrt\lambda}    \sqrt\lambda A \, \ud u
\\
   \nonumber
   & = \sqrt\lambda (\sqrt\lambda A + qI - (1/\sqrt\lambda) A Q)^{-1}
   A
 \\
    \label{e:p1st}
   &=  I - \frac{1}{\sqrt\lambda} q A^{-1} + O(1/\lambda)
 \end{align}
We repeat the proof
of \cite[Lemma 3.4]{ln13}, replacing $Q$ by $Q-q I$, and obtain
that
\begin{equation}
   \label{e:psilambda}
\Psi_\lambda(q) = I + \frac{1}{\sqrt\lambda} \Theta U(q) +
O(1/\lambda)
\end{equation}
for $q \geq 0$, where $U(q)$ is as stated in the lemma.
Altogether, the transition matrix is
\begin{align}
   \nonumber
\Phi_\lambda & = [I -(I- \frac{1}{\sqrt\lambda} q A^{-1})(I + \frac{1}{\sqrt\lambda} \Theta U(q))]^{-1}
\\  
   \nonumber 
 & \qquad
\times  [\frac{1}{\sqrt\lambda} q A^{-1} +(I - \frac{1}{\sqrt\lambda}
 q A^{-1}) \frac{1}{\sqrt\lambda} \Theta (U - U(q))]
 + O(1/\lambda)
\\ @
   \label{e:phil}
& =  I + (q A^{-1} - \Theta U(q))^{-1} \Theta U +  O(1/\sqrt\lambda),
\end{align}
which converges to the stochastic matrix
\[
\Phi = I + (q A^{-1} - \Theta U(q))^{-1} \Theta U
\]
as $\lambda \rightarrow \infty$.   The matrices $\Phi$ and
$\Phi_\lambda$ are irreducible and so the stationary probability
vector $\vrho_\lambda$ of $\Phi_\lambda$ converges to the stationary
probability vector $\vrho$ of $\Phi$, from which (\ref{e:rhost}) follows.

The remainder of the proof is immediate --- observe that
(\ref{e:nust}) is meaningful as $U$ is an irreducible generator.
\end{proof}

\begin{rem} \em
The matrix $U(q)$ defined in Lemma \ref{t:stickyboundary} has the
following physical interpretation: define $t_x = \inf\{t: X(t) \geq
x\}$.  we have 
\[
(e^{U(q)x})_{ij} = \P[t_x < E_q, \varphi(t_x) = j  |  \varphi(0)=i]
\]
where $E_q$ is an exponentially distributed random variable with
parameter $q$.
In other words, $U(q)$ is the generator of the Markov process
$\{\varphi(t_x)\}$ if the MMBM is killed at the exponential time
$E_q$.  Equation (\ref{e:quadraticU}) is a particular case of
\cite[Eqn (2.2)]{ivano10}.
\end{rem}

\begin{lem}
   \label{t:stickyTime}
   The expected time spent by the process $\{Y_\lambda(t), \bar\varphi_\lambda(t)\}$
   in the closed interval $[0,x]$ between regeneration points
   converges, as $\lambda \rightarrow \infty$, to
\begin{equation}
   \label{e:bigm}
M(x)   = (q A^{-1} -\Theta U(q))^{-1} \, (A^{-1} +2 
(-K)^{-1} (I-e^{K x} )\Theta^{-1})   
\end{equation}
for $x \geq 0$,  where 
\begin{equation}
   \label{e:KnU}
K = \Theta U \Theta^{-1} + 2 \Theta^{-2} D.
\end{equation}
The expected inter-regeneration
  interval is 
\begin{equation}
   \label{e:mgras}
 \vect m   =   (q A^{-1} -\Theta U(q))^{-1} \, (A^{-1} +2  (-K)^{-1}
 \Theta^{-1})\vone. 
\end{equation}
\end{lem}

\begin{proof}
  We denote by $M_\lambda(x)$ the expected time spent by the fluid
  queue in $[0,x]$ and we examine $M_\lambda(0)$ first.  After a
  regeneration, the process remains at level 0 either until
  the timer expires, or until  $\bar\kappa_\lambda$ switches from 2 to 1.  If
  the timer expires first, the sojourn time at 0 is over,
  otherwise, the fluid begins to grow and we wait for the level to
  return to 0.  If the return to 0 happens after the timer has expired, the
  sojourn time at 0 is over, otherwise an additional interval at 0
  begins.  Thus,
\begin{align*}
M_\lambda (0) & = (\sqrt\lambda A + q I - (1/\sqrt\lambda) AQ)^{-1} 
   + P_{\lambda,1} \Psi _\lambda  (q) M_\lambda (0) \\
  & = (I - P_{\lambda,1} \Psi _\lambda (q))^{-1} 
(\sqrt\lambda A + q I - (1/\sqrt\lambda) A Q)^{-1}
\\
 & = (\frac{1}{\sqrt\lambda} (q A^{-1} -\Theta U(q))
      + O(1/\lambda))^{-1}
(\frac{1}{\sqrt\lambda} A^{-1} +O(1/\lambda))
\end{align*}
by (\ref{e:p1st}, \ref{e:psilambda}).  This converges to $(qA^{-1}
-\Theta U(q))^{-1} A^{-1}$ 
as $\lambda \rightarrow \infty$.   

For $x > 0$, we have by a similar decomposition
\begin{equation}
   \label{e:Mla}
M_\lambda(x)  = (\sqrt\lambda A + q I - (1/\sqrt\lambda) AQ)^{-1} 
  + P_{\lambda,1} M_f(x)
  + P_{\lambda,1} \Psi_\lambda(q) M_\lambda(x),
\end{equation}
where $M_f(x)$ is the matrix of expected time spent in the semi-open
interval $(0,x]$ until the next return to 0, irrespective of the
timer being off or not at the time of return.   We rewrite (\ref{e:Mla}) as 
\begin{align}
   \nonumber
M_\lambda(x)  & = (I - P_{\lambda,1} \Psi_\lambda(q))^{-1}
 \,  ((\sqrt\lambda A + q I - (1/\sqrt\lambda) AQ)^{-1} 
        + P_{\lambda,1} M_f(x))
\\
   \label{e:Ml}
  & = M_\lambda(0) + (I - P_{\lambda,1} \Psi_\lambda(q))^{-1}P_{\lambda,1} M_f(x).
\end{align}
By (\ref{e:p1st}, \ref{e:psilambda}), we have
\[
(I - P_{\lambda,1} \Psi(q))^{-1} P_{\lambda,1}  = \sqrt\lambda
((-U(q))^{-1} \Theta^{-1} + O(1/\sqrt\lambda)
\]
and by \cite[Lemma 3.6]{ln13} and \cite[Theorem 3.7]{ln15a}, 
\[
M_f(x) = \frac{2}{\sqrt\lambda} (-K)^{-1}  (I - e^{Kx}) \Theta^{-1} + O(1/\lambda)
\]
where $K$ is given by (\ref{e:KnU}).   
Altogether, this shows that the limit of $M_\lambda(x)$ is given by
(\ref{e:bigm}).  
The proof of (\ref{e:mgras}) is immediate.
\end{proof}

We collect Lemmas \ref{t:stickyboundary} and \ref{t:stickyTime} in the
theorem below and obtain two formally different expressions for the
stationary distribution of the MMBM with sticky boundary.  The first
one directly follows from our regenerative process approach, the
second is independent of the parameter $q$.  In particular, one may
verify that (\ref{e:distnuB}) is identical to (\ref{e:bmsticky}) when
there is only one phase.

\begin{thm}
   \label{t:stickydist}
The stationary probability distribution function of the MMBM with sticky
boundary at zero is given by
\begin{equation}
   \label{e:distrhoB}
\vG(x) = \gamma_\rho \, \vect\rho
( q A^{-1}  -\Theta U(q) )^{-1} 
(A^{-1} + 2 (- K)^{-1} (I-e^{K x}) \Theta^{-1}),
\end{equation}
where $\vect\rho$ is the solution of the system (\ref{e:rhost}) and
$\gamma_\rho = 2(\vect\rho \vm)^{-1}$ is the normalization constant.

The distribution is also given by 
\begin{equation}
   \label{e:distnuB}
\vG (x) = \gamma_\nu \vect\nu(A^{-1} + 2 (- K)^{-1} (I-e^{K x}) \Theta^{-1}),
\end{equation}
independently of $q$, where $\vect\nu$ is the solution of the
system (\ref{e:nust}) and 
$\gamma_\nu = (\vnu (A^{-1} + 2 (- \Theta K)^{-1} )\vone)^{-1}$ is
the normalizing constant.  
\cqfd
\end{thm}

\begin{proof}
The process $\{Y_\lambda(t), \bar\kappa_\lambda(t),
\bar\varphi_\lambda(t)\}$ has the same distribution as the process
$\{Z_\lambda(\Gamma_\lambda(t)), \kappa_\lambda(\Gamma_\lambda(t)),
\varphi_\lambda(\Gamma_\lambda(t))\}$.  By \cite[Theorem 2.7]{ln13},
$\{Z_\lambda(t), \varphi_\lambda(t)\}$ weakly converges to $\{Z(t),
\varphi(t)\}$ and so, by the continuity and convergence properties of
$\Gamma(t)$, we find that the finite-dimensional distributions of
$\{Y_\lambda(t), \bar\varphi_\lambda(t)\}$ converge to the finite-dimensional
distribution of $\{Y(t), \bar\varphi(t)\}$.

Furthermore, the family $\{Y_\lambda(t), \bar\varphi_\lambda(t)\}$ is tight.  To
see this, we adapt the proof of \cite[Theorem 2.6]{ln13} and use
Lemma~\ref{t:Phi}.

Finally, we adapt the proof of \cite[Theorem 3.1]{ln14} to conclude
that the stationary distribution of $\{Y_\lambda(t), \bar\varphi_\lambda(t)\}$
converge to the stationary distribution of $\{Y(t), \bar\varphi(t)\}$.
Together with Lemmas \ref{t:stickyboundary} and
\ref{t:stickyTime}, this completes the proof.
\\ \hspace*{1em}
\end{proof}

The presence of the factor $A^{-1}$ in the expression for $M(0)$ is easy to understand: the greater
$a_i$, the faster the process leaves level 0 and the smaller the
mass at zero for phase $i$.  The marginal distribution of the phase is
no longer equal to $\valpha$, as we show in corollary \ref{t:margsticky}
below; its proof is immediate and is omitted.

\begin{cor}
   \label{t:margsticky}
The marginal distribution  of the phase is 
\[
\vG(\infty) = \gamma_\nu
\vect\nu (A^{-1} +2 (-\Theta K)^{-1}).
\]
\cqfd
\end{cor}

\section{Markov-regenerative MMBM}
\label{s:regenerative}

We revisit here the standard  MMBM $\{Z(t), \varphi(t)\}$  defined in
(\ref{e:Z}) and follow our regenerative process approach to determine
its stationary distribution.  Although expressions are known already
for the stationary distribution (Rogers~\cite{roger94},
Asmussen~\cite{asmus95b}, Latouche and Nguyen~\cite{ln13}), this new
analysis is of independent interest because it is one of the first to
analyze the MMBM as a regenerative process.  Harrison~\cite[Chapter 5,
Section 4]{harri90} does treat the regulated Brownian motion with two
boundaries as a regenerative process, but we take a different path.

We follow the same steps as in Section~\ref{s:sticky-mmbms} and, to
avoid confusion with the results there, we use the mark ``$*$''
in the present section.  Thus, $\vrho^*_\lambda$ and
$M^*_\lambda(x)$ represent, respectively, the stationary distribution of
the phase at regeneration epochs, and the expected time in $[0,x]$
between regenerations, for the flip-flop process $\{Z_\lambda(t),
\kappa_\lambda(t), \varphi_\lambda(t)\}$ with generator (\ref{e:qstar}).

\begin{lem}
   \label{t:boundary}
As $\lambda \rightarrow \infty$, $\vect\rho^* _\lambda$ converges to
$\vect\rho^*$ such that
\begin{equation}
   \label{e:rho}
\vect\rho^* (-U(q))^{-1} U = \vzero,  \qquad \vect\rho^*
  \vone = 1,
\end{equation}
where $U(q)$ is defined in Lemma~\ref{t:stickyboundary}.
In addition, 
\begin{equation}
   \label{e:rhoNnu}
\vect\rho^* (-U(q))^{-1}  = c_2 \vect\nu \Theta,
\end{equation}
for some scalar $c_2$, where  $\vnu$  is characterized by (\ref{e:nust}).
\end{lem}

\begin{proof}
  We start from Lemma \ref{t:Phi} and we repeat the argument in the proof of Lemma
  \ref{t:stickyboundary}, the only difference being that the matrices
  $P_{\lambda,0}$ and $P_{\lambda,1}$ are given here by
\begin{align}
   \label{e:p0}
P_{\lambda,0} & = q (\lambda I + q I - Q)^{-1} = O(1/\lambda),  \\
   \label{e:p1}
P_{\lambda,1} & = \lambda (\lambda I + q I - Q)^{-1} = I + O(1/\lambda),
\end{align}
so that the matrix $\Phi_\lambda$ from (\ref{e:phi}) is
\[
\Phi_\lambda = I - U(q)^{-1} U + O(1/\sqrt\lambda)
\]
and converge to $\Phi = I - U(q)^{-1} U$.  The remainder of the proof
is straightforward.
\end{proof}

Our next step is to determine the expected time spent in $[0,x]$
during a regeneration interval, and then collect the pieces in
Theorem~\ref{t:distribution}.

\begin{lem}
   \label{t:expectedTime}
The expected time spent by the MMBM in level 0 between regeneration
points is 0.  The time spent in $[0,x]$ (or equivalently in $(0,x]$) is
\begin{equation}
   \label{e:Mofx}
M^*(x)   = 2 (-U(q))^{-1} \Theta^{-1} (-K)^{-1} (I-e^{K x}
)\Theta^{-1}   \qquad \mbox{for $x \geq 0$.}
\end{equation}
The expected length of an interval between regenerations is 
\begin{equation}
   \label{e:m}
 \vect m^*  = 2 (-U(q))^{-1} \Theta^{-1}  (-K)^{-1} \Theta^{-1}\vone.
\end{equation}
\end{lem}
\begin{proof}
We follow the same steps as in the proof of Lemma
\ref{t:stickyTime}.   The expected time at level zero is 
\begin{align}
   \nonumber
M^* _\lambda (0) & = (\lambda I + q I - Q)^{-1} + P_{\lambda,1} \Psi _\lambda
(q) M^*_\lambda (0) \\
   \nonumber
  & = (I - P_{\lambda,1} \Psi _\lambda (q))^{-1} (\lambda I + q I -
  Q)^{-1} 
\\
   \label{e:Mstar0}
 & = \frac{1}{\sqrt\lambda} (- U(q))^{-1} \Theta^{-1} + O(1/\lambda)
\end{align}
and so,  $\lim_{\lambda \rightarrow \infty}  M^* _\lambda (0) = 0$.
For strictly positive $x$, we have
\[
M^* _\lambda (x)  = M^*_\lambda(0) +  (I - P_{\lambda,1} \Psi(q))^{-1}   P_{\lambda,1} M^*_f(x)
\]
instead of (\ref{e:Ml}), and (\ref{e:Mofx}) follows after simple manipulations.  The proof of
(\ref{e:m}) is immediate.
\end{proof}

The next theorem directly follows from Lemmas \ref{t:boundary} and
\ref{t:expectedTime} and is given without proof.

\begin{thm}
   \label{t:distribution}
The stationary probability distribution function of the regulated MMBM is given by
\begin{equation}
   \label{e:distrho}
\vG^*(x) = \gamma^*_\rho \, \vect\rho^*
(-U(q))^{-1}  \Theta^{-1}
(-K)^{-1} (I-e^{K x}) \Theta^{-1},
\end{equation}
where $\vect\rho^*$ is the solution of $\vrho^* (-U(q))^{-1} U =
\vzero$, $\vrho^* \vone = 1$, and $\gamma^*_\rho = 2 (\vrho^* \vm^*)^{-1}$ is
the normalizing constant.

It is also given by 
\begin{equation}
   \label{e:distnu}
\vG^*(x) = \gamma^*_\nu \vect\nu(- K)^{-1} (I-e^{K x}) \Theta^{-1},
\end{equation}
independently of $q$, where $\vect\nu$ is the solution of
the system
$\vect\nu \Theta U = \vect 0$, $\vect\nu \vect 1 = 1$, and 
$\gamma^*_\nu= (\vnu (-K)^{-1} \Theta^{-1} \vone)^{-1}$.
\cqfd
\end{thm}

The vector $\vnu$ is proportional to the vector $\vect\zeta_1$ defined
in \cite[Theorem 3.7]{ln13} and so the expression (\ref{e:distnu}) is nearly identical to the one given there.

\section{Observations}
\label{s:observation}

The equations (\ref{e:distnuB}, \ref{e:distnu}) have the advantage
over (\ref{e:distrhoB}, \ref{e:distrho}) of being independent of the
artificial parameter $q$.  On the other hand, the vectors $\vrho$ and
$\vrho^*$ have the physical meaning of being the stationary
distribution of the phase at epochs of regeneration, while the
interpretation of $\vnu$ is not as clear, as we discuss below.

We define the set $\EE = \{\theta_n: n \geq 0\}$ of regeneration
epochs, and we partition it into three disjoint subsets:
\begin{align*}
\EE_0 & = \{\theta_n: \theta_n = \theta_{n-1} + \Delta_n < \tau_{n-1},
n \geq 1\}  
\\ 
\EE_1 & = \{\theta_n: \theta_n = \theta_{n-1} + \Delta_n > \tau_{n-1},
n \geq 1\}  
\\ 
\EE_a & = \{\theta_n: \theta_n > \theta_{n-1} + \Delta_n, n \geq 1\}.
\end{align*}
If $\theta_n$ is in $\EE_0$ or in $\EE_1$, it means that the fluid is
equal to zero when the timer expires; in the first case, it has not
left level 0 at all between $\theta_{n-1}$ and $\theta_n$, in the
second case, the fluid has left level 0 and has returned there,
possibly several times.
If  $\theta_n$ is in $\EE_a$, then the fluid
is above level 0 when the timer expires. To keep
the notation simple, we do not indicate that these sets depend on
$\lambda$.

We also define $\EE^\u  = \{\theta_n^\u\}$ to be the set of {\em all} epochs when
the fluid hits level 0 from above: starting from $\theta_0^\u =0$, we define
\begin{align*}
 \tau_n^\u &= \inf\{t > \theta_n^\u : \varphi_\lambda(t) \in \s_\u \},  \\
\theta_{n+1}^\u & = \inf\{t > \tau_n^\u : Z_\lambda(t)=0\}.
\end{align*}
Clearly, $\EE_a \subset \EE^\u$, and $\EE_b$ defined as $\EE_b = \EE^\u
\setminus \EE_a$ is the set of all epochs when the process returns to
0 from above before the expiration of the timer.

By definition, $\vrho_\lambda$ (as well as $\vrho_\lambda^*$) is the
limiting distribution of $\varphi_\lambda(t)$ as $t$ goes to infinity
by taking values in $\EE_0 \cup \EE_1 \cup \EE_a$, while we see from
(\ref{e:nulambda}) that $\vect\nu_\lambda$ is the limiting
distribution as $t$ goes to infinity by taking values in $\EE_0 \cup
\EE^\u =\EE_0 \cup \EE_b \cup \EE_a $.
This observation provides us with  a
physical interpretation for 
(\ref{e:rnlambda}): we rewrite that equation as
\[
\vect\rho_\lambda  (I- P_{\lambda,1} \Psi _\lambda (q))^{-1}= c
\vect\nu _\lambda
\]
and we note that $(I-P_{\lambda,1} \Psi_\lambda(q))^{-1}$ is the
matrix of expected
number of returns to level 0 at epochs in $\EE_b$ between two
successive regeneration points.

We now focus on the traditional MMBM analyzed in
Section~\ref{s:regenerative}, and we compare $\vnu_\lambda$ and
$\vrho^*_\lambda$. In this case, $\EE_0$ and $\EE_1$ vanish as
$\lambda$ grows bigger, the vector $\vect\nu_\lambda$ becomes more
like the stationary distribution of the phase at {\em all epochs} when the
fluid returns to level 0, irrespective of the timer, while
$\vect\rho^*_\lambda$ becomes more like the stationary distribution at
the {\em subset} of those epochs when we have an actual regeneration.
In the limit, the interpretation of $\vnu_\lambda$ may not be given
as such to $\vnu$ due to the
instantaneous repeated hits at the boundary by the Brownian motion,
once it reaches level 0.

The vectors $\vrho^*$ and $\vnu$ are 
both related to the regulator 
 $R(t)$
of $\{X(t)\}$.   Recall that
$R(t) = |\inf_{0 \leq s \leq t} X(s)|$ is split into  the
sub-regulators defined
in~(\ref{e:ri}).   It is shown in Asmussen and
Kella~\cite{ao00} that $r_i(t)/t$ converges almost surely as $t
\rightarrow \infty$, and we define
\begin{equation}
   \label{e:l}
\ell_{ i}   = \lim_{t \rightarrow \infty}  r_i(t)/t.
\end{equation}
Similarly, the regulator $R_\lambda(t) = |\inf_{0 \leq s \leq t}
X_\lambda(s)|$ of the fluid queue is split into the sub-regulators
$r_{\lambda,i}(t)$ in (\ref{e:ril}) 
and we define $\ell_{\lambda, i}   = \lim_{t \rightarrow \infty} r_{\lambda,i}(t)/t$.

We proved in \cite{ln14, ln13} the weak convergence of
$\{Z_\lambda(t), \varphi_\lambda(t)\}$ to $\{Z(t), \varphi(t)\}$. In
consequence, the functions $R_\lambda(t)$, $t \geq 0$, weakly converge
to $R(t)$ and the vectors $\vect\ell_\lambda$ converge to $\vect\ell$,
as $\lambda \rightarrow \infty$.

The function $r_{\lambda,i}(t)$
increases at the rate $|\mu_i - \sqrt\lambda \sigma_i|$ during those
intervals of time when $(Z_\lambda(t), \varphi_\lambda(t))=(0,i)$, and
so
\begin{align*}
\ell_{\lambda,i} &= (\sigma_i \sqrt\lambda - \mu_i)(\vrho^*_\lambda
\vm^*_\lambda)^{-1} (\vrho^*_\lambda M^*_\lambda(0))_i
\\
 & = (\sigma_i \sqrt\lambda - \mu_i)  (\vrho^*\vm^* +
 O(1/\sqrt\lambda))^{-1}  (\frac{1}{\sqrt\lambda} \vrho^*
 (-U(q))^{-1} \Theta^{-1} + O(1/\lambda))_i
\\
 & = (\vrho^*\vm^*)^{-1}  (\vrho^*  (-U(q))^{-1})_i + O(1/\sqrt\lambda)
\end{align*}
from which we obtain
\begin{align}
   \label{e:lnrho}
\vect\ell & = (\vrho^*\vm^*)^{-1}  (\vrho^*  (-U(q))^{-1})
\\
   \label{e:lnnu}
 & = c_3 \vnu \Theta
\end{align}
for some scalar $c_3$ by (\ref{e:rhoNnu}).
This provides us with another expression for the stationary distribution of
MMBMs.

\begin{cor}
   \label{t:distl}
The stationary probability distribution function of the MMBM is given by
\begin{equation}
   \label{e:distl}
\vG^*(x) = 2 \vect\ell \Theta^{-1}(-K)^{-1} (I - e^{Kx}) \Theta^{-1}
\end{equation}
where $\vect\ell$ is defined in (\ref{e:l}).  The vector $\vect\ell$
is the solution of the linear system 
\begin{equation}
   \label{e:lnU}
\vect\ell U = \vzero, \qquad 2\vect\ell \Theta^{-1}(-K)^{-1}
\Theta^{-1} \vone =1.
\end{equation}
\end{cor}

\begin{proof}
  Equation (\ref{e:distl}) is a direct consequence of
  (\ref{e:distrho}) and of the relation~(\ref{e:lnrho}) between
  $\vect\ell$ and $\vrho^*$.  Also, by (\ref{e:lnnu}), we find that
  $\vect\ell U = \vzero$ since $\vnu \Theta U = \vzero$ by~(\ref{e:nust}).
  Finally, we use the normalizing equation $\vG^*(\infty) \vone = 1$
  and the proof is complete.
\end{proof}

One last representation of the stationary distribution $\vG^*$
establishes a direct connection with the stationary distribution
$\valpha$ of the Markov process $\{\varphi(t)\}$ with generator $Q$.

\begin{cor}
   \label{t:distq}
The stationary probability distribution function of the MMBM is given by
\begin{equation}
   \label{e:distq}
\vG^*(x) = \valpha \Theta (I - e^{Kx}) \Theta^{-1}
\end{equation}
where $\valpha$ is the solution of the system $\valpha Q = \vzero$, $\valpha \vone
=1$. 
\end{cor}

\begin{proof}
Obviously, the marginal distribution of the phase
is $\valpha$, so that by~(\ref{e:distnu})
\begin{equation}
   \label{e:qnnu}
\valpha = \vG^*(\infty) = \gamma_\nu^* \vnu (-K)^{-1} \Theta^{-1}
\end{equation}
and (\ref{e:distnu}) may be rewritten as (\ref{e:distq}).  
\end{proof}

Corollary \ref{t:distq} may also be proved by a purely algebraic
argument.  We give it below as it illustrates the intricate
interconnection between different matrices.   We proceed through the
sequence of equations
\begin{align*}
&& \valpha Q  & = \vzero  \\
\Leftrightarrow &&\valpha (\Theta^2 U^2 + 2 D U)  & = \vzero  \qquad \mbox{by
  (\ref{e:quadraticU}) evaluated at $q =0$,}\\
\Leftrightarrow &&\valpha \Theta K \Theta U  & = \vzero  \qquad \mbox{by (\ref{e:KnU})}
\\
\Leftrightarrow &&\valpha \Theta K & = c_4 \vect\nu
\end{align*}
for some scalar $c_4$ by (\ref{e:nust}).  Therefore, (\ref{e:distnu}) becomes
\[
\vG^*(x) = c_5 \valpha \Theta (I-e^{Kx}) \Theta^{-1}
\]
for some scalar $c_5$, and it is easily seen that $c_5=1$ since
$\vG^*(\infty) \vone = 1$.

\begin{rem} \em
Clearly, the stationary distribution of MMBMs may be expressed under
many different guises, even without counting the ones based on the
time-reversed process, as in \cite{asmus95b, roger94}.  
We find the matrix $(I-e^{Kx}) \Theta^{-1}$ in each case,  pre-multiplied by
vectors which depend on the behavior of the process at the boundary.



The connections (\ref{e:distl}) with the vector $\vect\ell$, and
(\ref{e:distq}) with the vector $\valpha$ crucially depend on the
evolution of the phase being independent of the fluid level.  Indeed,
both $\vect\ell$ and $\valpha$ are defined by the unrestricted MMBM:
\begin{itemize}
\item[(a)]
The vector $\valpha$ is the stationary marginal distribution of the phase when its
evolution is governed by the matrix $Q$ and is not modified in any
way; this is the key to the proof of Corollary \ref{t:distq}.  
\item[(b)]
The vector $\vect\ell$ is defined in (\ref{e:l}) as the rate of
increase of the regulator in the absence of any barrier.
Corollary~\ref{t:distl} requires that the stationary distribution at
regenerations be related to $\vect\ell$ through (\ref{e:lnrho}), which
is not true of the process analyzed in
Section~\ref{s:sticky-mmbms} and the one defined in the next section.
\end{itemize}
\end{rem}

\section{Resampling the phase}
\label{s:resampling}

In Section~\ref{s:sticky-mmbms}, we slow down the evolution of the
process at level 0 and we use different factors $a_i$ for different
phases, but the behavior of the phase is not otherwise
modified.  We go one step further now and allow for
more general perturbations.  In the generator (\ref{e:q0})
of the fluid queue $\{Y_\lambda(t), \bar \kappa_\lambda(t),
\bar\varphi_\lambda(t)\}$ at level 0, transitions of
$\bar\kappa_\lambda$ from 2
to 1 occur at rates proportional to $\sqrt\lambda$ while transitions
of $\bar\varphi_\lambda$ occur at much smaller rates of order
$1/\sqrt\lambda$.  We shall now assume that both
$\bar\kappa_\lambda$ and $\bar\varphi_\lambda$ may evolve at rates
proportional to $\sqrt\lambda$.

We define a new family $\{\widetilde Y_\lambda(t),
\widetilde\kappa_\lambda(t), \widetilde\varphi_\lambda(t)\}$ of fluid queues
with generator $Q_\lambda^*$ given in~(\ref{e:qstar}) when $\widetilde Y_\lambda > 0$, and
generator
\begin{equation}
   \label{e:qt0}
\widetilde Q_{\lambda,0} =  \vligne{\sqrt\lambda A  &  (1/\sqrt\lambda)
  \widetilde Q + \sqrt\lambda \widetilde A}
\end{equation}
when $\widetilde Y_\lambda = 0$.  That is, simultaneous transitions are
possible from $(\bar\kappa_\lambda,\bar\varphi_\lambda) =(2,i)$ to
$(1,j)$ for $i \not=j$ at the rate $\sqrt\lambda A_{ij}$.  In
(\ref{e:q0}), $A$ is a diagonal matrix, $\widetilde A = -A$, and $\widetilde Q
= A Q$.

\begin{ass}
   \label{a:B}
We assume that $A \geq 0$, $\widetilde A_{ij} \geq 0$ for $i \not= j$,
$\widetilde A_{ii} < 0$, $i$, $j$ in $\MM$, and that $A + \widetilde
A$ is an irreducible 
generator.  The matrix $\widetilde Q$ is such that $(1/\sqrt\lambda)
  \widetilde Q + \sqrt\lambda \widetilde A$ is an irreducible generator for
  $\lambda$ large enough.
\end{ass}

In consequence, the matrix $B$ defined as $B = (-\widetilde A)^{-1} A$
is stochastic and irreducible, and we denote its stationary
probability vector as $\vbeta$. The assumption that $A + \widetilde A$
is irreducible is a significant restriction: by contrast,
$A+\widetilde A =0$ for the transition matrix (\ref{e:q0}), and $B$ is
the identity matrix. We make this assumption so as to simplify the
presentation of the process and to let its major feature stand out.

The matrix $\widetilde Q$ plays a minor role since the speed of changes
induced by $ (1/\sqrt\lambda)  \widetilde Q$ is negligible with respect to
$\sqrt\lambda$.  Actually, our expression in Theorem
\ref{t:stickydistB} for the limiting  stationary
distribution does not depend on $\widetilde Q$.  

Because of the additional mixing of the phases allowed by the matrices
$A$ and $\wA$, we have a more complex transformation than the simple
change of clock in Section~\ref{s:sticky-mmbms}.  Away from 0,
$\{\widetilde Y_\lambda(t), \widetilde \varphi_\lambda(t)\}$ behaves exactly
like $\{Z_\lambda(t), \varphi_\lambda(t)\}$ but at level 0, as $\lambda$ grows
bigger, the evolution of the phase is controlled mostly through the
matrices $\sqrt\lambda \wA$ and $\sqrt\lambda A$ and, in first
approximation, the distribution of the phase is repeatedly transformed
by the transition matrix $B$ upon each visit to the boundary.

\begin{lem}
   \label{t:stickyB}
The stationary distribution of $\widetilde\varphi_\lambda$ at epochs
of regeneration converges as $\lambda \rightarrow \infty$ to
$\wrho$ with
\begin{equation}
   \label{e:stickyrhoB}
\wrho  = \widetilde\gamma \vect\beta( q (-\wA)^{-1}  + \Theta(U- U(q) )), 
\end{equation}
where $\vect\beta$ is the stationary probability vector of $B$ and
\begin{equation}
   \label{e:gammatilde}
\widetilde\gamma = (\vect\beta( q (-\wA)^{-1}  - \Theta U(q)
)\vone)^{-1}
\end{equation}
is the normalization constant.
\end{lem}

\begin{proof}
The transition matrix at epochs of regeneration is given by
(\ref{e:phi1}), where
\begin{align*}
P_{\lambda,0} & = q (-\sqrt\lambda \wA + q I - (1/\sqrt\lambda) \wQ)^{-1}  \\
  & = \frac{1}{\sqrt\lambda} q (-\wA)^{-1} + O(1/\lambda)
\\[0.5\baselineskip]
P_{\lambda,1} & = \sqrt\lambda(-\sqrt\lambda \wA + q I - (1/\sqrt\lambda) \wQ)^{-1} A \\
  & = B + \frac{1}{\sqrt\lambda} q \wA^{-1} B
+ O(1/\lambda).
\end{align*}
Taking the limit as $\lambda \rightarrow \infty$ on both sides of
(\ref{e:phi1}), we find that the limit $\Phi$ of $\Phi_\lambda$
satisfies the equation $\Phi = B \Phi$ where $B$ is stochastic.
This shows that $\Phi$ is of rank one, and that
\begin{equation}
   \label{e:phitilde}
\Phi = \vone \cdot \wrho
\end{equation}
for some vector $\wrho$.  The matrix $\Phi$ is stochastic, and so
$\wrho$ is its stationary
probability vector.
Thus, 
\[
\Phi_\lambda = \vone \cdot \wrho + \frac{1}{\sqrt\lambda} \Phi' +
O(1/\lambda) 
\]
for some matrix $\Phi'$ and, by equating the coefficients of
$1/\sqrt\lambda$ on both sides of (\ref{e:phi1}), we get
\[
\Phi' = q (-\wA)^{-1}  + B \Theta (U - U(q)) 
 + B \Phi' + B \Theta U(q) \vone \cdot \wrho
+ q\wA^{-1} B \vone \cdot \wrho.
\]
since $B \vone = \vone$.  
We pre-multiply both sides by $\vect\beta$ and obtain 
\[
\vzero =   \vect\beta( q (-\wA)^{-1}  + \Theta(U- U(q) ))  + (\vect\beta \Theta U(q) \vone + q \vect \beta \wA^{-1}
\vone) \wrho  
\]
from which (\ref{e:stickyrhoB}, \ref{e:gammatilde}) follow.
\end{proof}

\begin{lem}
   \label{t:stickyTimeB}
In the limit as $\lambda\rightarrow \infty$, the expected time spent
in $[0,x]$  between regeneration points is 
\begin{align}
   \label{e:Mtilde}
\wM(x)  & =  \widetilde\gamma \, \vone\cdot\vect\beta
(-\wA^{-1} +2 (-K)^{-1} (I-e^{K x} )\Theta^{-1})   
\intertext{for $x \geq 0$,  and}
   \label{e:smtilde}
 \widetilde{\vect m}  & =  \widetilde\gamma
\,(\vect\beta (-\wA^{-1} +2 (-K)^{-1} \Theta^{-1})   \vone)
\, \vone.
\end{align}
\end{lem}
\begin{proof}
We decompose $\wM_\lambda(x)$  as in Lemmas
\ref{t:stickyTime} and  \ref{t:expectedTime}: 
\begin{align*}
\wM_\lambda(x)  & = (- \sqrt\lambda \wA + q I -
\frac{1}{\sqrt\lambda}\wQ)^{-1}  
  + P_{\lambda,1} M_f(x) + P_{\lambda,1} \Psi_\lambda(q) \wM_\lambda (x) 
\\
  & = 
\frac{1}{\sqrt\lambda} (-\wA)^{-1}  + (B + \frac{1}{\sqrt\lambda}
q \wA^{-1} M_f(x)  )
\\ & \quad
+ (B + \frac{1}{\sqrt\lambda} q \wA^{-1} (I +
\frac{1}{\sqrt\lambda} \Theta U(q)) \wM_\lambda (x)
+ O(1/\lambda).
\end{align*}
In the limit, $\wM_\lambda(x)$ converges to a solution of
$\wM(x) = B \wM(x)$, so that
\[
\lim_{\lambda\rightarrow\infty} \wM_\lambda(x) = \vone \cdot \widetilde{\vect\mu}(x)
\]
for some vector $\widetilde{\vect\mu}(x)$ which needs to be determined.
We proceed just like we did in the proof of Lemma~\ref{t:stickyB}, and
obtain
\[
\widetilde{\vect\mu}(x) = \widetilde\gamma
\vect\beta \,
(A^{-1} +2 (-K)^{-1} (I-e^{K x} )\Theta^{-1})   .
\]
The proof of (\ref{e:smtilde}) is immediate.
\end{proof}

\begin{rem} 
   \label{r:jiggle}
\em We observe in (\ref{e:phitilde}) the
  effect of the Brownian motion jiggle at the boundary: by the time
  the exponential timer is off, the process will have hit level 0 so
  often that the phase at the next regeneration epoch will be
  independent of the phase at the last one.

  The same effect is at work in (\ref{e:Mtilde}): after a regeneration
  point, the phase will be repeatedly re-sampled through the matrix
  $B$, so often  that the expected length of any interval between
  regeneration points is independent of the phase at the end of the
  previous interval,  and depends only on the stationary distribution of
  $B$.   

In short, we refer to the limit as a process with sticky boundary and
resampling of the phase at level zero.
\end{rem}

\begin{thm}
   \label{t:stickydistB}
   The stationary probability distribution function of the process
   $\{\widetilde Y(t), \widetilde \varphi(t)\}$ with sticky boundary and
   resampling of the phase at zero is given by
\begin{equation}
   \label{e:distbeta}
\widetilde \vG (x) = \widetilde\gamma \vect\beta(-\wA^{-1} + 2 (- K)^{-1} (I-e^{K x}) \Theta^{-1}),
\end{equation}
independently of $q$, where $\vect\beta$ is the stationary
probability vector of $B$ and 
$\widetilde\gamma$ is given in (\ref{e:gammatilde}).

The marginal distribution of the phase is $\widetilde \vG(\infty) = \widetilde\gamma
\vect\beta(-\wA^{-1} + 2 (- K)^{-1} \Theta^{-1})$.
\end{thm}

\begin{proof}
By Lemmas \ref{t:stickyB} and \ref{t:stickyTimeB}, 
\begin{align*}
\widetilde \vG (x) & = 
\widetilde\gamma \,\vect\beta( -q \wA^{-1}  + \Theta(U- U(q) )) \ 
\widetilde\gamma
\vone \cdot \vect\beta
(- \wA^{-1} +2 (-K)^{-1} (I-e^{K x} )\Theta^{-1})   
\\ & =
\widetilde\gamma \, \vect\beta
(- \wA^{-1} +2 (-K)^{-1} (I-e^{K x} )\Theta^{-1})   
\end{align*}
after reorganization of some factors, and using the relation $U \vone = \vzero$.
\end{proof}

\section{Concluding remarks}
\label{s:conclusion}

To the best of our knowledge, these are the first results on MMBMs
where the evolution of the phase $\varphi$ may depend on the level,
with the exception of Chen {\it et al.} \cite{cht02}: the authors
consider MMBMs with level-dependent, piecewise constant, fluid rates
and obtain the stationary distribution by numerically solving
systems of partial differential equations.

Our regenerative approach to the analysis of regulated MMBMs clearly
shows great promise in allowing more complex assumptions than has been
the case until now.  We have demonstrated this on two specific cases
of reactive boundaries in Sections~\ref{s:sticky-mmbms} and
\ref{s:resampling} but other examples easily come to mind, as in
Latouche and Nguyen~\cite{ln15a}.

In each case covered here, the stationary distribution is easily
calculated once the matrices $U$ and $U(q)$ are determined.  Extremely
efficient algorithms exist to solve the matrix equation
(\ref{e:quadraticU}), such as those in Latouche and
Nguyen~\cite{ln13},  and Nguyen and Poloni~\cite{np15}, 
and so the question of numerically
obtaining these distributions is not an issue.    \nocite{np15}

There is a striking difference between the ``traditional'' process
analyzed in Section~\ref{s:regenerative} and the two processes with
sticky boundary: the representation of the stationary distribution in
Corollaries \ref{t:distl} and 
\ref{t:distq} hold in the first case but not in the other two.
Instead, the importance of the distribution at epochs of regeneration
is made more manifest, observe that the vector $\vnu$ in~(\ref{e:distnuB}) and
$\vbeta$ in (\ref{e:distbeta}) are related in the same manner to the
distribution at epochs of regeneration: (\ref{e:rhonust}) may be
written as
\[
c_1 \vect \nu A^{-1} = \vect\rho M(0)
\]
and (\ref{e:stickyrhoB}) as
\[
c_6 \vect \beta (-\widetilde A)^{-1} = \wrho \wM(0)
\]
for some scalar $c_6$.
In both equations, the $i$th component of the vector in the right-hand
side is the expected time spent in phase $i$ at level 0 between two
regeneration points.

\subsubsection*{Acknowledgements}
The authors thank the Minist\`ere de la
Communaut\'e fran\c{c}aise de Belgique for supporting  this research
through the ARC grant AUWB-08/13--ULB~5, they acknowledge the
financial support of the Australian Research Council through the 
Discovery Grant DP110101663.

\bibliographystyle{abbrv}
\bibliography{regenMMBM}

\end{document}